\pdfoutput=1
\documentclass[%
  a4paper,%
  11pt,%
  abstract=off,%
  headings=small,%
]{scrartcl}

\usepackage[ngerman,english]{babel}
\usepackage[T1]{fontenc}
\usepackage[utf8]{inputenc}
\usepackage{textcomp}
\usepackage[left=23mm,right=23mm,top=25mm,bottom=40mm]{geometry}

\usepackage{mathptmx} \usepackage[scaled]{helvet} \usepackage{courier}

\usepackage{amsmath,amssymb,amsthm}
\usepackage{overpic}
\usepackage{graphicx}
\usepackage[outline]{contour}
\contourlength{0.7pt}
\usepackage{multicol}
\usepackage{microtype}
\usepackage{ellipsis}
\usepackage{xcolor}
\definecolor{uibkorange}{rgb}{1,0.5,0}
\definecolor{uibkblue}{rgb}{0,0.18,0.44}

\usepackage{babelbib}
\usepackage{enumitem}
\setlist{noitemsep,leftmargin=*}
\usepackage[runin]{abstract}
\usepackage{todonotes}
\usepackage{bookmark}
\usepackage{hyperref}
\usepackage[all]{hypcap}

\setkomafont{sectioning}{\normalfont\bfseries}
\setcapindent{2em}

\renewcommand*{\vec}[1]{\mathbf{#1}}

\newcommand*{\RSet}{\mathbb{R}}
\newcommand*{\ZSet}{\mathbb{Z}}
\newcommand*{\Norm}[1]{\Vert #1 \Vert}

\newcommand*{\defeq}{\mathrel{\mathop:}=}

\newcommand*{\SQ}{\mathcal{S}}
\newcommand*{\SP}[2]{\langle #1, #2 \rangle}
\newcommand*{\dunit}{\varepsilon}
\newcommand*{\sv}{\mathfrak{S}}
\newcommand*{\vv}{\mathfrak{V}}

\newcommand*{\primal}[1]{\widetilde{#1}}
\newcommand*{\dual}[1]{\widehat{#1}}

\newcommand*{\quaternion}[1]{\mathrm{#1}}
\newcommand*{\qi}{\quaternion{i}}
\newcommand*{\qj}{\quaternion{j}}
\newcommand*{\qk}{\quaternion{k}}
\newcommand*{\quadrilateral}[4]{[#1,#2,#3,#4]}

\newtheorem{theorem}{Theorem}
\newtheorem{proposition}[theorem]{Proposition}

\newtheorem{corollary}[theorem]{Corollary}
\theoremstyle{definition}
\newtheorem{definition}[theorem]{Definition}
\theoremstyle{remark}

\title{Discrete Gliding Along Principal Curves}
\author{Hans-Peter Schröcker\thanks{Hans-Peter Schröcker, Unit Geometry and CAD, University
  Innsbruck, Technikerstraße~13, A6020~Innsbruck, Austria,
  \url{http://geometrie.uibk.ac.at/schroecker/}}}
\date{\today}

\hypersetup{
  pdfauthor={Hans-Peter Schröcker},%
  pdftitle={Discrete Gliding Along Principal Curves},%
  pdfkeywords={discrete differential geometry, kinematics, rotational motion, rotation net, curvature line discretization, principal contact element net, gliding motion},%
  pdfsubject={},%
  pdfstartview={FitH},%
  pdfproducer={},%
  pdfcreator={},%
  pdfborder={0 0 0},%
  colorlinks=false,%
}

\begin{document}

\maketitle%
\begin{onecolabstract}
We consider $n$-dimensional discrete motions such that any two
neighbouring positions correspond in a pure rotation (``rotating
motions''). In the Study quadric model of Euclidean displacements
these motions correspond to quadrilateral nets with edges contained in
the Study quadric (``rotation nets''). The main focus of our
investigation lies on the relation between rotation nets and discrete
principal contact element nets. We show that every principal contact
element net occurs in infinitely many ways as trajectory of a discrete
rotating motion (a discrete gliding motion on the underlying
surface). Moreover, we construct discrete rotating motions with two
non-parallel principal contact element net trajectories. Rotation nets
with this property can be consistently extended to higher dimensions.



%
{\small
\medskip\par\noindent
\emph{Keywords:} Discrete differential geometry,
          kinematics,
          rotational motion,
          rotation net,
          curvature line discretization,
          principal contact element net,
          gliding motion.
\par\noindent
\emph{MSC 2010:}
53A17, 
53A05, 
52C35  
}

\end{onecolabstract}

\begin{multicols}{2}
\section{Introduction}
\label{sec:introduction}

Discrete differential geometry is an active field of geometrical
research. Its aim is the development of discrete notions for
well-known concepts from differential geometry. The resulting theories
are often more elementary and concrete when compared to their smooth
counterparts. While classic differential geometry is largely based on
analysis, elementary geometric incidence or closure theorems are at
the core of discrete differential geometry. It is a discipline that
naturally lends itself to applications that require numeric
simulation, visualization, or the building of real world objects. An
excellent introduction to the current state of research is the
monograph \cite{bobenko08:_discrete_differential_geometry}.

In this article we relate recent progress in the theory of discrete
curvature line parametrizations to spatial kinematics. We study
discrete nets of proper (orientation-preserving) Euclidean
displacements such that neighbouring positions correspond in a
relative rotation (``discrete rotating motions''). Smooth motions with
this property naturally arise as gliding motions along principally
parametrized surfaces. Their discrete counterparts, rotating motions
with discrete curvature line trajectories, are the main topic of this
article.

In Section~\ref{sec:preliminaries} we recall the notion of principal
contact element nets\,---\,families of contact elements (point plus
oriented tangent plane), indexed by $\ZSet^n$, such that neighbouring
contact elements have a common tangent sphere. Principal contact
element nets have been introduced in
\cite{bobenko07:_organizing_principles} as a comprehensive concept
that captures different notions of discrete principal parametrizations
(circular nets and conical nets, see
\cite[Section~3.1]{bobenko08:_discrete_differential_geometry} and
\cite{pottmann08:_focal_circular_conical}). Indeed, the points of a
principal contact element net form a circular net (the elementary
quadrilaterals are circular; Figure~\ref{fig:pce-quadrilateral},
left), while its planes form a conical net (four planes meeting in a
vertex are tangent to a cone of revolution;
Figure~\ref{fig:pce-quadrilateral}, center).

A concise formulation of all calculations and formulas in this article
is possible within the dual quaternion calculus of spatial kinematics
(Section~\ref{sec:kinematics-dual-quaternions}). In
Section~\ref{sec:rotatation-quadrilaterals} we recall fundamental
results on ``rotation quadrilaterals''
\cite{schroecker10:_four_positions_theory}, the elementary building
blocks of rotation nets.

\begin{figure*}
  \begin{minipage}{0.33\linewidth}
    \centering
    \begin{overpic}{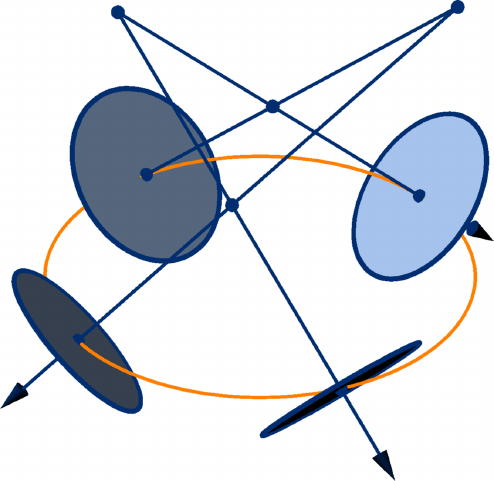}
      \put(3,10){\textcolor{uibkblue}{$(p_{\vec{i}}, \nu_{\vec{i}})$}}
      \put(72,12){\textcolor{uibkblue}{$\tau_1(p_{\vec{i}}, \nu_{\vec{i}})$}}
      \put(0,60){\textcolor{uibkblue}{\contour{white}{$\tau_2(p_{\vec{i}}, \nu_{\vec{i}})$}}}
      \put(50,53){\textcolor{uibkblue}{$\vec{z}^1_{\vec{i}}$}}
      \put(92,88){\textcolor{uibkblue}{$\vec{z}^2_{\vec{i}}$}}
    \end{overpic}
  \end{minipage}%
  \hfill
  \begin{minipage}{0.33\linewidth}
    \centering
    \includegraphics{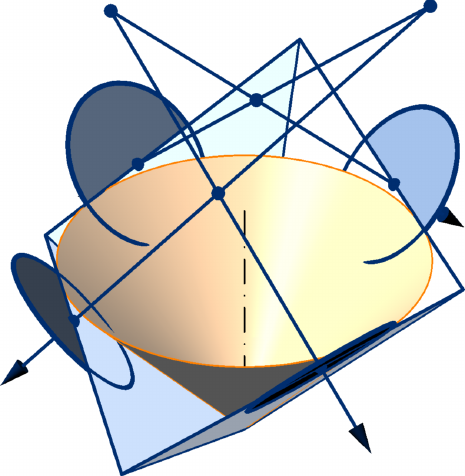}
  \end{minipage}%
  \hfill
  \begin{minipage}{0.33\linewidth}
    \centering
    \includegraphics{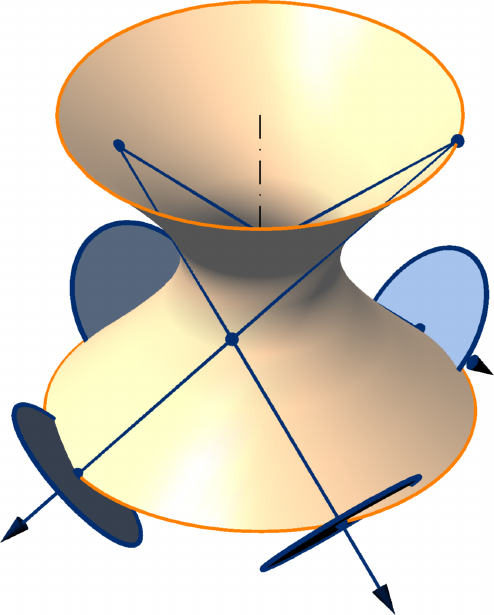}
  \end{minipage}
  \caption{Elementary quadrilateral in a principal contact
    element net: Points on a circle (left), planes tangent
    to a cone of revolution (middle), normal lines on a hyperboloid of
  revolution (right)}
  \label{fig:pce-quadrilateral}
\end{figure*}

The major contributions of this article are presented in
Section~\ref{sec:discrete-rotating-motions}. Just as a smooth surface
gives rise to many gliding motions, parametrized by principal lines
and the rotation angle about surface normals, a principal contact
element net occurs in many ways as trajectory of a discrete rotating
motion. The main result is a proof of the multidimensional consistency
of discrete rotating motions with two independent principal contact
element trajectories. The two trajectory surface are related by
discrete version of the classic Bäcklund transform for pseudospherical
surfaces.

\section{Preliminaries}
\label{sec:preliminaries}

\subsection{Curvature line discretizations}
\label{sec:curvature-line-discretizations}

A parametrization of a smooth surface in $\RSet^3$ is called a
curvature line parametrization or principal parametrization if
infinitesimally neighbouring surface normals along both families of
parameter lines intersect. Generically, this parametrization is unique
(up to re-parametrization of the individual parameter lines). The
property of concurrent neighbouring normals is preserved in its usual
discretizations. The most prominent example of discrete curvature
lines are circular nets\,---\,quadrilateral nets such that any
elementary quadrilateral has a circumcircle, see for example
\cite[Section~3.1]{bobenko08:_discrete_differential_geometry}. An
alternative discretization are conical nets\,---\,quadrilateral nets
such that the planes meeting in a vertex are tangent to a cone of
revolution \cite{pottmann08:_focal_circular_conical}. Neighboring
circle axis and neighbouring cone axis in circular and conical nets
intersect and can serve as discrete surface normals. In
\cite{bobenko07:_organizing_principles} principal contact element nets
were introduced as a generalization of both, circular and conical
nets.

\begin{definition}
  \label{def:contact-element-net}
  An \emph{oriented contact element} is a pair $(p,\nu)$ consisting of
  a point $p$ and an oriented plane $\nu$ incident with $p$. The
  oriented line $N$ orthogonal to $\nu$ and incident with $p$ is
  called the \emph{axis} or \emph{normal} of the contact element, $p$
  is its \emph{vertex} and $\nu$ its \emph{tangent plane}.
\end{definition}

Contact element nets are quadrilateral nets of oriented contact
elements, that is, they are maps from $\ZSet^n$ to the space of
oriented contact elements. The image of $\vec{i} \in \ZSet^n$ is
denoted by $(p_{\vec{i}}, \nu_{\vec{i}})$. In
Definition~\ref{def:principal-contact-element-net} below we adopt the
notation of \cite{bobenko08:_discrete_differential_geometry} where
$\tau_i$ indicates a shift of indices in the $i$-th coordinate
direction. For example $\tau_1p_{120} = p_{220}$, $\tau_2p_{120} =
p_{130}$, etc.

\begin{definition}
  \label{def:principal-contact-element-net}
  A \emph{contact element net} is a map $\vec{i} \mapsto (p_{\vec{i}},
  \nu_{\vec{i}})$ from $\ZSet^n$ to the space of oriented contact
  elements. A \emph{principal contact element net}
  (Figure~\ref{fig:pce-quadrilateral}) is a contact element net such
  that any two neighbouring contact elements
  $(p_{\vec{i}},\nu_{\vec{i}})$, $\tau_i(p_{\vec{i}},\nu_{\vec{i}})$
  have a common oriented tangent sphere.
\end{definition}

In other words, neighbouring normals $N_{\vec{i}}$ and
$\tau_iN_{\vec{i}}$ in a principal contact element net intersect in a
point $z^i_{\vec{i}} = N_{\vec{i}} \cap \tau_iN_{\vec{i}}$ which is at
the same oriented distance from $p_{\vec{i}}$ and~$\tau_ip_{\vec{i}}$.
The points $p_{\vec{i}}$ in a principal contact element net constitute
the vertices of a circular net, the oriented planes $\nu_{\vec{i}}$
are the face planes of a conical net. The contact element axes
$N_{\vec{i}}$ form a discrete line congruence
\cite[Section~2.2]{bobenko08:_discrete_differential_geometry} such
that any elementary quadrilateral lies on a hyperboloid of revolution
(Figure~\ref{fig:pce-quadrilateral}, right). Discrete line congruences
of this type have interesting properties but have not yet been
discussed in literature.

In this article we consider a kinematic generation of principal
contact element nets in $\RSet^3$. A concise analytic description can
be obtained by means of the dual quaternion calculus of spatial
kinematics which shall be introduced now.

\subsection{Kinematics and dual quaternions}
\label{sec:kinematics-dual-quaternions}

Quaternions and dual quaternions are important tools in theoretical
and applied kinematics, see for example the description in
\cite{husty10:_algebraic_geometry_kinematics},
\cite[Section~4.5]{mccarthy90:_introduction_theoretical_kinematics} or
\cite[Chapter~9]{selig05:_fundamentals_robotics}. In our study they
turn out to be a versatile tool as well. Equations in dual quaternion
form are concise, of low degree, accessible to geometric
interpretations and free of the need of case distinctions. We assume
that the reader is familiar with quaternion algebra (see for example
\cite{husty10:_algebraic_geometry_kinematics} or
\cite[Chapter~4]{mccarthy90:_introduction_theoretical_kinematics}) and
describe only the extension to dual quaternions and its application to
spatial kinematics.

A dual quaternion is an object of the form $a = \primal{a} +
\dunit\dual{a}$ where primal part $\primal{a}$ and scalar part
$\dual{a}$ are ordinary quaternions and $\dunit$ is the dual unit
satisfying $\dunit^2 = 0$. The addition of dual quaternions is
performed component-wise for primal and dual parts. The dual
quaternion multiplication $\star$ extends the multiplication of
ordinary quaternions. It is associative, distributive and the
quaternion units $1$, $\qi$, $\qj$, $\qk$ commute with $\dunit$.
These properties define the dual quaternion multiplication uniquely.

A dual quaternion $a = \primal{a} + \dunit\dual{a}$ can be identified
with a vector $a = (\primal{a}, \dual{a}) = (a_0,\ldots,a_7)$ in
$\RSet^8$. \emph{Vector part} and \emph{dual part} of $a$ are
\begin{equation}
  \label{eq:1}
  \begin{aligned}
    \vv a &= (0,a_1,a_2,a_3,0,a_5,a_6,a_7) = \vv{\primal{a}} + \dunit\vv{\dual{a}},\\
    \sv a &= (a_0,0,0,0,a_4,0,0,0) = \sv{\primal{a}} + \dunit\sv{\dual{a}}.
  \end{aligned}
\end{equation}
A normalized dual quaternion $a = \primal{a} + \dunit\dual{a}$
satisfies two conditions:
\begin{equation}
  \label{eq:2}
  \Norm{\primal{a}}^2 = a_0^2 + a_1^2 + a_2^2 + a_3^2 = 1
  \quad\text{and}\quad
  \SP{a}{a} = 0
\end{equation}
where
\begin{equation}
  \label{eq:3}
  \SP{a}{b} \defeq \sum_{i=0}^3 (a_ib_{i+4} + a_{i+4}b_i).
\end{equation}
The second equation in \eqref{eq:2} is the well-known Study
condition.

Plücker coordinate vectors $P = (p_0,\ldots,p_5)$ for straight lines
(see for example \cite[Chapter~2]{pottmann01:_line_geometry}) can be
embedded in the space of dual quaternions as
\begin{equation}
  \label{eq:4}
  P \hookrightarrow (0, p_0, p_1, p_2, 0, p_3, p_4, p_5).
\end{equation}
In this case the Study condition \eqref{eq:2} reduces to the Plücker
condition such that normalized Plücker coordinates (characterized by
$p_0^2 + p_1^2 + p_2^2 = 1$) become normalized dual quaternions of
vanishing scalar part. In this article we do not distinguish between a
straight line and its Plücker coordinates, embedded in the space of
dual quaternions.

Normalized dual quaternions with non-vanishing primal part constitute
a two-fold covering of the group of proper Euclidean displacements.
The action of a normalized dual quaternion $a$ on a point with
coordinate vector $(x_1,x_2,x_3)$ and a line $P$ (both in the moving
space) is given by
\begin{equation}
  \label{eq:5}
  \begin{aligned}
    1 + \dunit x' & = a_\dunit \star (1 + \dunit x) \star a^{-1}, \\
    P' & = a_\dunit \star P \star a_\dunit^{-1},
  \end{aligned}
\end{equation}
where $x = (0,x_1,x_2,x_3)$ is a vector valued ordinary quaternion,
$a_\dunit \defeq \primal{a} - \dunit\dual{a}$, and the prime indicates
coordinates in the fixed space.  The inverse displacement is described
by the inverse dual quaternion $a^{-1}$ defined by the equation $a
\star a^{-1} = 1$. It is uniquely defined for all dual quaternions
with non-vanishing primal part.

Since $a$ and $-a$ describe the same displacement, it is natural to
identify proportional dual-quaternions, thus arriving at Study's
kinematic mapping which associates proper Euclidean displacements with
points of the Study quadric
\begin{equation}
  \label{eq:6}
  \SQ\colon \SP{x}{x} = 0.
\end{equation}
The Study quadric is a hyperquadric in the seven-dimensional
projective space $P^7$ over $\RSet^8$. Only points of an exceptional
three-space $E$ with equation $\primal{x} = (0,0,0,0)$ do not occur as
images of proper Euclidean displacements.  We always identify
Euclidean displacements with homogeneous coordinate vectors that
describe points on the Study quadric (minus~$E$).

\subsection{Rotation quadrilaterals}
\label{sec:rotatation-quadrilaterals}

In Section~\ref{sec:discrete-rotating-motions} we will consider
special quadrilateral nets in the Study quadric. The geometry of an
elementary quadrilateral
\begin{equation}
  \label{eq:7}
  \quadrilateral{a_{\vec{i}}}{\tau_ia_{\vec{i}}}{\tau_ja_{\vec{i}}}{\tau_i\tau_ja_{\vec{i}}}
  = \quadrilateral{a_0}{a_1}{a_2}{a_3}
\end{equation}
of a net of this type is discussed in this section. It follows the
presentation of \cite{schroecker10:_four_positions_theory}.

\begin{definition}
  A quadrilateral $\quadrilateral{a_0}{a_1}{a_2}{a_3}$ on the Study
  quadric is called a \emph{rotation quadrilateral} if its vertices
  and edges are contained in~$\SQ$.
\end{definition}

The name ``rotation quadrilateral'' is justified by the observation
that the edge through $a_i$ and $a_{i+1}$ is contained in $\SQ$ if and
only if the relative displacement $r_{i,i+1} \defeq a_{i+1} \star
a_i^{-1}$ is a pure rotation (indices modulo four; see
\cite[Satz~19]{weiss35:_liniengeometrie_und_kinematik} or
\cite{husty10:_algebraic_geometry_kinematics}). (We will frequently
use results and formulas of
\cite{weiss35:_liniengeometrie_und_kinematik} but, occasionally, adapt
them to match our convention of quaternion multiplication which is
slightly different from that used
in~\cite{weiss35:_liniengeometrie_und_kinematik}.) The algebraic
characterization of relative rotations is
\begin{equation}
  \label{eq:8}
  \pi_5(r_{i,i+1}) = \pi_5(a_{i+1} \star a_i^{-1}) = 0
\end{equation}
where $\pi_5(x)$ is the projection onto the fifth coordinate (the dual
scalar part) of a dual quaternion
\cite[Satz~13]{weiss35:_liniengeometrie_und_kinematik}. We denote the
relative revolute axis of $r_{i,i+1}$ in the moving space
by~$R_{i,i+1}$.

The main result of \cite{schroecker10:_four_positions_theory}
characterizes contact elements whose homologous images form an
elementary quadrilateral in a principal contact element net.

\begin{proposition}[\cite{schroecker10:_four_positions_theory}]
  \label{prop:rotation-quadrilateral}
  The only contact elements in the moving space whose homologous
  images with respect to a generic rotation quadrilateral form a
  non-degenerate elementary quadrilateral of a principal contact
  element net are those whose axes are transversal to the four
  relative revolute axes $R_{01}$, $R_{12}$, $R_{23}$, and~$R_{30}$.
\end{proposition}

Further elementary quadrilaterals of principal contact element nets
have the normal $R_{i,i+1}$. Since their $a_i$ and $a_{i+1}$ images
are identical, these quadrilaterals are degenerate. For a generic
rotation quadrilateral there exist two (possibly complex or
coinciding) transversals $M$ and $N$ of the four relative revolute
axis. An important property of rotation quadrilaterals is stated in
the following completion theorem:

\begin{theorem}
  \label{th:completion}
  Consider two skew lines $M$, $N$ in the moving space and two
  displacements $a_0$, $a_2$ of a rotation quadrilateral. Generically,
  there exist two positions $a_1$, $a_3$ (possibly complex) such that
  $\quadrilateral{a_0}{a_1}{a_2}{a_3}$ is a rotation quadrilateral
  with relative revolute axes $R_{i,i+1}$ that intersect $M$ and~$N$.
\end{theorem}

\begin{proof}
  We consider a (yet undetermined) position $x$ such that the relative
  displacements $x \star a_0^{-1}$, $x \star a_2^{-1}$ are rotations
  whose axes $X_0$, $X_2$ intersects $M$ and $N$. It turns out that
  this problem amounts to solving a quadratic equation from which the
  Theorem's claim follows.

  The Plücker line coordinate vector of the relative rotation axis
  $X'_i$ in the fixed space is
  \begin{equation}
    \label{eq:9}
    X'_i = \vv(x_i)_\dunit
  \end{equation}
  (adapted from
  \cite[Satz~13]{weiss35:_liniengeometrie_und_kinematik}).  According
  to \eqref{eq:5}, the Plücker coordinate vector of the relative
  revolute axis $X_i$ in the moving space is
  \begin{equation}
    \label{eq:10}
      X_i = (a_i)^{-1}_\dunit \star X'_i \star (a_i)_\dunit.
  \end{equation}
  Hence, the sought position $x$ has to satisfy the six linear
  equations
  \begin{equation}
    \label{eq:11}
    \begin{gathered}
      \SP{(a_i)_\dunit \star \vv(x \star a_i^{-1})_\dunit \star
        (a_i)_\dunit^{-1}}{T}= 0,\\
      \pi_5(x \star a_i^{-1}) = 0
    \end{gathered}
  \end{equation}
  with $i \in \{0, 2\}$ and $T \in \{M, N\}$, and the quadratic
  equation $\SP{x}{x} = 0$. The solutions are the intersection points
  of a straight line with the Study quadric \eqref{eq:6}.
\end{proof}

We will actually need the result of Theorem~\ref{th:completion} in the
following form:

\begin{corollary}
  \label{cor:completion}
  It is possible to complete a rotation quadrilateral whose relative
  revolute axes intersect two skew lines $M$, $N$ in the moving space
  from three admissible input positions $a_0$, $a_1$, $a_2$. The input
  data is admissible if the two relative displacements $a_1 \star
  a_0^{-1}$ and $a_2 \star a_1^{-1}$ are rotations whose axes
  intersect $M$ and $N$. In this case, the missing position $a_3$ is
  unique and real (if the input positions are real).
\end{corollary}

\section{Discrete rotating motions}
\label{sec:discrete-rotating-motions}

Now we are ready to introduce the central concept of this article:

\begin{definition}
  A \emph{rotation net} (or a \emph{discrete rotating motion}) is a
  quadrilateral net whose vertices and edges are contained in the
  Study quadric.
\end{definition}

The elementary quadrilaterals of rotation nets are rotation
quadrilaterals. Rotation nets of dimension two discretize
two-parameter motions $x(t_1,t_2)$ with parameter lines $t_i =
\text{const.}$ whose instantaneous screws have vanishing pitch, that
is, they are actually rotations. The geometric interpretation in terms
of the Study quadric is that not only the point $x(t_1,t_2)$ but also
the tangents to the parameter lines are contained in the Study
quadric. We call parametrized motions of this type \emph{rotating} as
well. The extension of this concept to more-dimensional motions is
obvious.

Important examples of two-dimensional rotating motions are obtained
from a principal parametrization $\vec{f}(t_1,t_2)$ of a surface $\Phi
\subset \RSet^3$. The \emph{Darboux frame} associated with this
parametrization is the trihedron with base $\vec{f}$ and legs
$\vec{u}_1$, $\vec{u}_2$, and $\vec{n}$ where
\begin{equation}
  \label{eq:12}
  \vec{u}_i = \frac{\partial \vec{f}/\partial t_i}{\Norm{\partial \vec{f}/\partial t_i}},\quad
  \vec{n} = \vec{u}_1 \times \vec{u}_2.
\end{equation}
The Darboux frame is orthonormal. As $t_1$ and $t_2$ vary in their
respective domains, it moves along the surface $\Phi$. The thus
defined motion $x(t_1,t_2)$ is rotating and we may call it the Darboux
motion of the principal parametrization. If we add the rotation angle
$w$ about the surface normal $\vec{n}(t_1,t_2)$ as a third parameter
the motion $x(t_1,t_2,w)$ is still rotating (see
\cite[Section~7.1.5]{pottmann01:_line_geometry}, in particular
Remark~7.1.16). It is called a \emph{gliding motion} along
$\Phi$. Similar results hold true for more-parameter motions. These
observations motivate our study of the relation between discrete
rotating motions and discrete principal curvature parametrizations.

\subsection{Principal contact element nets and rotating motions}
\label{principal_contact_element_nets_and_rotating_motions}

We are going to discuss possibilities for generating a principal
contact element net $(p_{\vec{i}}, \nu_{\vec{i}})$ as trajectory of a
discrete rotating motion. Any two neighbouring contact elements
$(p_{\vec{i}}, \nu_{\vec{i}})$ and $\tau_i(p_{\vec{i}},
\nu_{\vec{i}})$ have a plane of symmetry $\beta_{\vec{i}}^i$ and,
starting from one contact element, the complete net can be generated
by successive reflections in the planes $\beta_{\vec{i}}^i$. The
sequence of reflections between two contact elements is not unique but
never leads to contradictions.

Sometimes (as for example in \cite{huhnen-venedey07:_curvature_line})
it is useful to assign an orthonormal frame
$(X_{\vec{i}},Y_{\vec{i}},Z_{\vec{i}})$ to every element of a
principal contact element net such that the origin coincides with
$p_{\vec{i}}$, $Z_{\vec{i}}$ is the normal line and the $X_{\vec{i}}$-
and $Y_{\vec{i}}$-axes correspond in the reflections at the planes
$\beta_{\vec{i}}^i$.  The thus obtained discrete line congruences
$X_{\vec{i}}$ and $Y_{\vec{i}}$ can be regarded as discretizations of
one family of principal curvature lines. It is obvious that by a
simple change of orientation of certain $X_{\vec{i}}$- and
$Y_{\vec{i}}$-axes the reflection at $\beta_{\vec{i}}^i$ can be
replaced by a rotation about an axis perpendicular to $N_{\vec{i}}$
and $\tau_iN_{\vec{i}}$ and incident with $z^i_{\vec{i}} = N_{\vec{i}}
\cap \tau_iN_{\vec{i}} \in \beta_{\vec{i}}^i$. The thus obtained
rotation net can be considered as discretization of a Darboux motion.
If we perturb every frame by a rotation through a certain angle about
its $Z_{\vec{i}}$-axis, neighbouring frames still correspond in a
rotation about an axis through $z^i_{\vec{i}}$ and in
$\beta_{\vec{i}}^i$.  This rotation net can be considered as a
$n$-dimensional discrete gliding motion along a principal
parametrization.

These considerations show that \emph{every principal contact element
  net occurs in multiple ways as trajectory of a rotation
  net}\,---\,just as any principal parametrization gives rise to
infinitely many gliding motions. The rotation angle about the
$Z_{\vec{i}}$-axis gives one degree of freedom per vertex.


\subsection{Pairs of principal contact element nets}
\label{sec:pairs-principal-contact-element-nets}

By Proposition~\ref{prop:rotation-quadrilateral}, rotation nets with
principal contact element nets as trajectories are characterized by
the fact that all relative rotation axis in the moving space intersect
a fixed line $N$ whose images constitute the set of contact element
normals. Clearly, every point $p \in N$ and the plane $\nu$ through
$p$ and orthogonal to $N$ define a contact element $(p,\nu)$ whose
trajectory is a principal contact element as well. In other words,
principal contact element nets as trajectories of discrete rotating
motions come in one-parameter families of parallel nets.

Calling two principal contact element nets \emph{independent} if there
exists a pair of contact elements with the same index $\vec{i} \in
\ZSet^n$ but different normals, we aim at a kinematic generation of
independent principal contact element nets as trajectories of rotation
nets. In view of Proposition~\ref{prop:rotation-quadrilateral} we can
hope at most for two independent trajectories. The characteristic
property is that all relative rotation axes in the moving space
intersect two fixed lines $M$ and $N$. We will show that such rotation
nets exist for arbitrary dimension of the underlying motion. This
result is rather surprising since a naive counting of free parameters
suggest only existence of $2$-dimensional motions with this property.

\begin{figure*}
  \centering
  \begin{minipage}[b]{0.5\linewidth-0.5\columnsep}
    \rule{\linewidth}{0pt}
    \centering
    \includegraphics{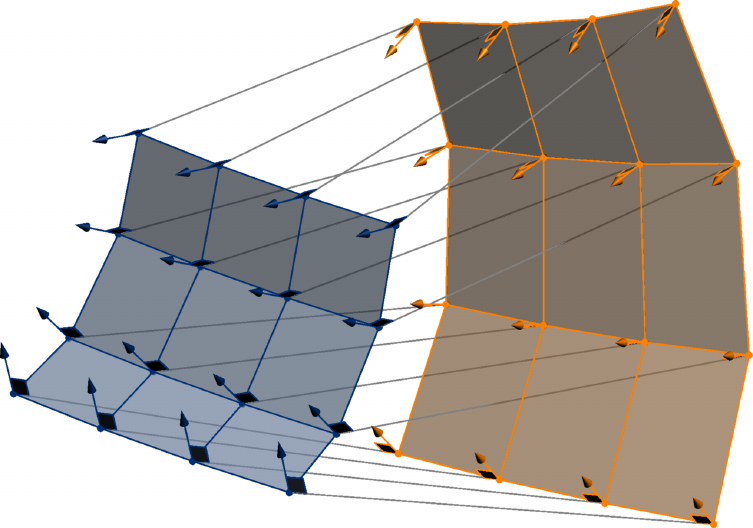}
    \caption{Two independent principal contact element net
      trajectories of a rotating motion}
    \label{fig:rotnet-04}
  \end{minipage}\hfill
  \begin{minipage}[b]{0.5\linewidth-0.5\columnsep}
    \rule{\linewidth}{0pt}
    \centering
    \begin{overpic}{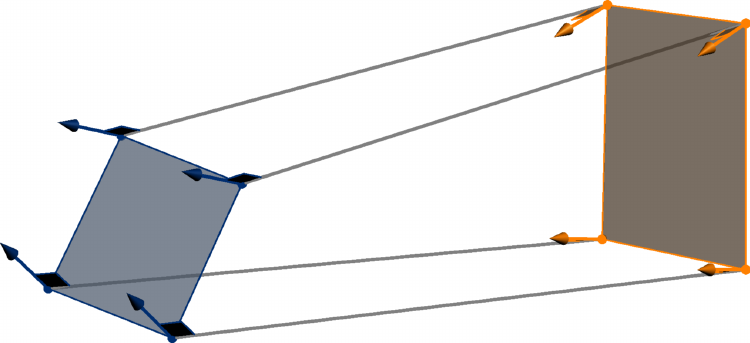}
      \put(3,3){\textcolor{uibkblue}{$p_0$}}
      \put(24,-3){\textcolor{uibkblue}{$p_1$}}
      \put(32,18){\textcolor{uibkblue}{$p_2$}}
      \put(12,31){\textcolor{uibkblue}{$p_3$}}
      \put(76,47){\textcolor{uibkorange}{$q_0$}}
      \put(77,10){\textcolor{uibkorange}{$q_1$}}
      \put(95,5){\textcolor{uibkorange}{$q_2$}}
      \put(98,45){\textcolor{uibkorange}{$q_3$}}
    \end{overpic}
    \hspace*{\baselineskip}
    \caption{The figures formed by corresponding contact elements are
      congruent}
    \label{fig:rotnet-04b}
  \end{minipage}
\end{figure*}

\begin{theorem}
  \label{th:extend}
  A discrete $n$-dimensional rotating motion with two independent
  principal contact element nets as trajectories is uniquely defined
  by two skew axes $M$, $N$ of the contact elements in the moving
  space and the values along the coordinate axes in~$\ZSet^n$.
\end{theorem}

For $n = 2$ this result follows directly from
Corollary~\ref{cor:completion}. Starting with $a_{00}$, $a_{10}$, and
$a_{01}$, the position $a_{11}$ is uniquely defined. From $a_{10}$,
$a_{20}$, and $a_{11}$ we can construct $a_{21}$, and so on. Two
independent principal contact element net trajectories of a discrete
rotating motion of dimension $n = 2$ are depicted in
Figure~\ref{fig:rotnet-04}. The figures formed by corresponding points
and face normals are congruent. Neighbouring figures correspond in a
pure rotation (Figure~\ref{fig:rotnet-04b}).

In case of $n \ge 3$ it is not immediately clear that this inductive
construction works. We describe the situation only for $n = 3$; the
problems in higher dimensions are similar. Consider the elementary
cube $a_{i\!jk}$ with $i$, $j$, $k \in \{0, 1\}$
(Figure~\ref{fig:3d-consistency}). The input data consists of
$a_{000}$, $a_{100}$, $a_{010}$, and $a_{001}$. According to
Corollary~\ref{cor:completion}, it defines $a_{110}$, $a_{101}$, and
$a_{011}$. Now there are three possibilities to construct the missing
vertex $a_{111}$, from the three positions $a_{100}$, $a_{110}$,
$a_{101}$, from the three positions $a_{010}$, $a_{110}$, $a_{011}$,
or from the three positions $a_{001}$, $a_{101}$, $a_{011}$. We have
to show that all thus constructed positions are actually identical. In
the terminology of \cite{bobenko08:_discrete_differential_geometry}
this is called \emph{3D-consistency} of rotation nets with two
independent principal curvature trajectories. In
Theorem~\ref{th:consistency}, below, we will show that these nets are
actually \emph{$n$D-consistent.} This is a fundamental property in the
discretization of differential geometric concepts and immediately
implies Theorem~\ref{th:extend}.

\begin{theorem}
  \label{th:consistency}
  Generic discrete rotation nets with two independent curvature line
  trajectories are $n$D-consistent.
\end{theorem}

\begin{figure*}
  \begin{minipage}[b]{0.50\linewidth}
    \rule{\linewidth}{0pt}
    \centering
    \includegraphics{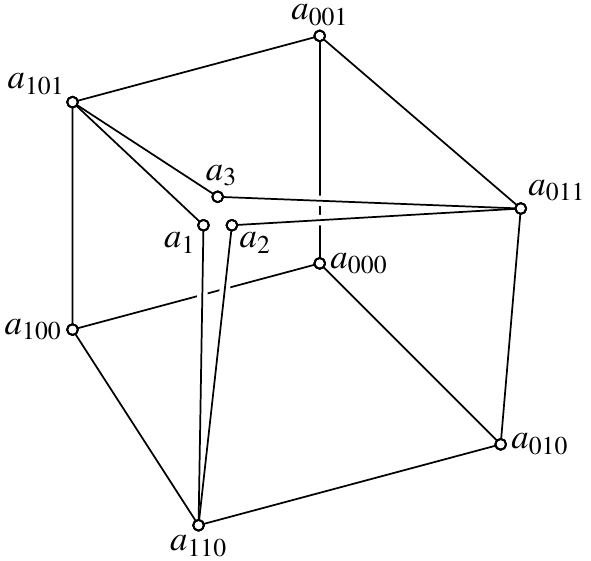}
    \caption{3D-consistency}
    \label{fig:3d-consistency}
  \end{minipage}%
  \begin{minipage}[b]{0.50\linewidth}
    \rule{\linewidth}{0pt}
    \centering
    \includegraphics{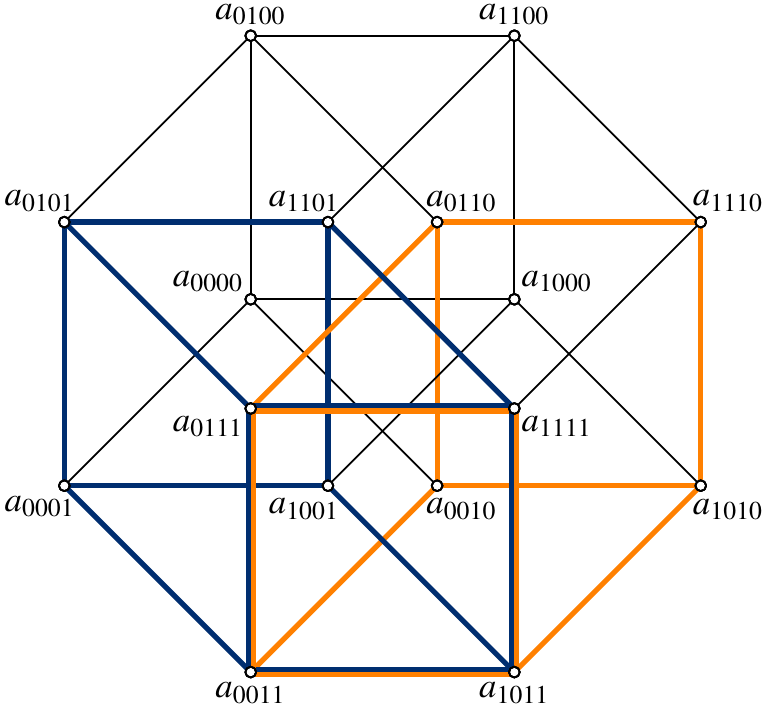}
    \caption{4D-consistency}
    \label{fig:4d-consistency}
  \end{minipage}%
\end{figure*}

\begin{proof}
  We have to show that a generic $n$-dimensional cube can be
  constructed from one vertex $a_{(0,\ldots,0)}$ and $n$-adjacent
  vertices $\tau_i a_{(0,\ldots,0)}$, $i \in \{1, \ldots, n\}$.  Using
  the notation of the last paragraph before the Theorem we proof
  3D-consistency at first.  We already know that there exist points
  $a_1$, $a_2$ and $a_3$ such that
  \begin{equation}
    \label{eq:13}
    \begin{aligned}
      &\quadrilateral{a_{100}}{a_{110}}{a_{101}}{a_1},\\
      &\quadrilateral{a_{010}}{a_{110}}{a_{011}}{a_2},\\
      &\quadrilateral{a_{001}}{a_{101}}{a_{011}}{a_3}
    \end{aligned}
  \end{equation}
  are rotation quadrilaterals with all required properties. We have to
  show that these points coincide whereupon we may set $a_{111} \defeq
  a_1 = a_2 = a_3$.

  If we require only one principal contact element net trajectory (for
  example with normal $M$ and plane $\mu \perp M$), only six linear
  equations remain. They define a straight line $K \subset P^7$ which,
  in the sense of algebraic geometry, has at least two intersection
  points $k_1$, $k_2$ with the Study quadric. Since principal contact
  element nets are known to be $n$D-consistent
  \cite{bobenko07:_organizing_principles}, the $k_1$- and $k_2$-images
  of the contact element $(p,\mu)$ coincide. We denote this
  contact element by $(p_K, \mu_K)$ and conclude that the relative
  displacement between $k_1$ and $k_2$ is a pure rotation. This
  implies that the line $K$ is actually contained in the Study
  quadric~$\SQ$.

  Denote by $L \subset \SQ$ the straight line of positions obtained in
  the same way but with $M$ replaced by $N$ and by $(p_L, \nu_L)$
  the corresponding contact element. The proof of 3D-consistency will
  be finished if we can show that $K$ and $L$ have a point in
  common. Assume conversely that $K$ and $L$ are skew. Then they span
  a three-space whose intersection with $\SQ$ is a ruled quadric. We
  conclude that every position $k \in K$ can be rotated to a unique
  position $l \in L$. In other words, the contact element $(p_K,
  \nu_K)$ can be rotated in infinitely many ways into the contact
  element $(p_L, \nu_L)$. By elementary geometric reasoning this
  is only possible if both contact elements have a common tangent
  sphere. This contradicts the skewness of $M$ and $N$. Hence, $k$ and
  $l$ intersect in a unique position $a_{111}$ which implies the
  Theorem's statement for $n = 3$.

  Now we consider 4D-consistency (see Figure~\ref{fig:4d-consistency}
  which displays the projection of a 4D-cube). According to
  Corollary~\ref{cor:completion}, the input data
  \begin{equation}
    \label{eq:14}
    a_{0000},\ a_{1000},\ a_{0100},\ a_{0010},\ a_{0001}
  \end{equation}
  defines the positions
  \begin{equation}
    \label{eq:15}
    a_{1100},\ a_{1010},\ a_{1001},\ a_{0110},\ a_{0101},\ a_{0011}.
  \end{equation}
  By the previous discussion, we can construct the position
  \begin{equation}
    \label{eq:16}
    a_{1110},\ a_{1101},\ a_{1011},\ a_{0111}
  \end{equation}
  without a contradiction. Now we have four ways to construct the
  missing position $a_{1111}$ by completing a 3D-cube.  In
  Figure~\ref{fig:4d-consistency} two of these cubes are
  highlighted. They share the common quadrilateral
  $\quadrilateral{a_{0111}}{a_{0011}}{a_{1011}}{a_{1111}}$. Thus, by
  Corollary~\ref{cor:completion}, completing the 3D-cubes leads to the
  same position for $a_{1111}$. This situation does not change if we
  consider other pairs of 3D-cubes.

  The same argument yield $n$D-consistency: Inductively it can be
  shown that the input data defines all positions uniquely with
  exception of one position $a_{\vec{1}}$. This position can be
  constructed by completing $n$ cubes of dimension $n-1$. But any two
  of these cubes share a face of dimension $n-2 \ge 2$ which already
  uniquely determines~$a_{\vec{1}}$.
\end{proof}

The construction of $n$-dimensional rotation nets as in
Theorem~\ref{th:extend} from the given input data is inductive. At
every step it requires solving the equation system \eqref{eq:11}
augmented with the Study condition \eqref{eq:2}. The solution is
generically unique and can be computed linearly. When prescribing the
positions on the coordinate axes only a subset of \eqref{eq:11} needs
to be fulfilled.

\section{Conclusion and future research}
\label{sec:conclusion-future-research}

In an attempt to relate recent progress in the field of discrete
differential geometry to kinematics this article introduces discrete
rotating motions and studies possibilities to obtain curvature line
discretizations as trajectories.

We are already in a position to present some of the implications of
our study. Assume that $(x_{\vec{i}}, \nu_{\vec{i}})$ and
$(y_{\vec{i}}, \mu_{\vec{i}})$ are independent principal contact
elements obtained as trajectory surfaces of a discrete rotating
motion. The figures formed by corresponding contact elements are
congruent (Figure~\ref{fig:rotnet-04b}) and clearly we can find two
families of parallel trajectory surfaces with the same properties.
Distinguished examples are obtained by choosing the contact points $x$
and $y$ on the common perpendicular of the contact element normals $M$
and $N$ in the moving space. In this case $x_{\vec{i}} \in
\mu_{\vec{i}}$ and $y_{\vec{i}} \in \nu_{\vec{i}}$ and the trajectory
surfaces satisfy all geometric properties of the classic Bäcklund
transform for pseudospherical surfaces (surfaces of constant Gaussian
curvature):
\begin{itemize}
\item Corresponding points are at constant distance (independent of
  $\vec{i} \in \ZSet^n$),
\item corresponding tangent planes intersect in a constant angle
  (independent of $\vec{i} \in \ZSet^n$), and
\item the connecting line of corresponding points lies in both tangent
  planes.
\end{itemize}

This observation leads to the conjecture that the trajectory surfaces
are discrete surfaces of constant Gaussian curvature in the sense of
\cite{bobenko10:_curvature_theory}. Indeed, we have a proof for this
which shall be published elsewhere. By a result of
\cite{bobenko10:_curvature_theory}, this implies that the parallel
trajectory surfaces are linear Weingarten surfaces. An analytic
description of a Bäcklund transform but for a different discrete
curvature line parametrization (``discrete $O$-surfaces'') can be
found in \cite{schief03:_unification}. Our main results, the
existence of discrete rotating motions with two independent principal
contact element net trajectories and their multidimensional
consistency, are meant to provide the basis for the description of
Bäcklund transforms of pseudospherical principal contact element
nets\,---\,in a forthcoming publication.








\end{multicols}
\end{document}